\documentclass[a4paper,12pt]{amsproc}

\usepackage{amssymb}
\usepackage{graphicx}
\usepackage{float}
\usepackage{amsmath}
\usepackage{array}
\usepackage{multirow}
\usepackage{mathrsfs}
\usepackage{textcomp}
\usepackage[bottom]{footmisc}
\usepackage{xcolor}
\usepackage{cite}

\newcolumntype{M}[1]{>{\raggedright}m{#1}}

\DeclareMathAlphabet{\mathpzc}{OT1}{pzc}{m}{it}

\newtheorem{theorem}{Theorem}[section]
\newtheorem{lemma}[theorem]{Lemma}
\newtheorem{proposition}[theorem]{Proposition}

\newtheorem{conjecture}[theorem]{Conjecture}

\theoremstyle{definition}
\newtheorem{definition}[theorem]{Definition}

\theoremstyle{remark}
\newtheorem{remark}[theorem]{Remark}

\numberwithin{equation}{section}

\allowdisplaybreaks[4]

\begin{document}

\null\vskip -1cm

\title{New $\mu$-Zariski pairs of surface singularities}

\author{Christophe Eyral, Masaharu Ishikawa, Mutsuo Oka, \"Oznur Turhan}

\address{C. Eyral, Institute of Mathematics, Polish Academy of Sciences, ul. \'Sniadeckich 8, 00-656 Warsaw, Poland}  
\email{cheyral@impan.pl} 

\address{M. Ishikawa, Department of Mathematics, Hiyoshi Campus, Keio University,  4-1-1 Hiyoshi, Kohoku-ku, Yokohama, Kanagawa 223-8521, Japan}
\email{ishikawa@keio.jp}

\address{M. Oka, Emeritus Professor, Tokyo Institute of Technology, 2-12-1 Ohokayama, Meguro-ku, Tokyo 152-8551, Japan}   
\email{okamutsuo@gmail.com}

\address{\"O. Turhan, Departement of Mathematics, Galatasaray University, Ortak\"{o}y 34357, Istanbul, Turkey and Institute of Mathematics, Polish Academy of Sciences, ul. \'Snia\-deckich~8, 00-656 Warsaw, Poland} 
\email{oturhan@impan.pl}

\thanks{}

\subjclass[2020]{14B05, 14J17, 32S05, 32S25}

\keywords{Complex surface singularity, $\mu$-Zariski pair, $\mu$-constant stratum, almost Newton non-degenerate function, monodromy zeta-function, dual resolution graph.}

\begin{abstract}
To the best of the authors' knowledge, all previously known examples of $\mu$-, $\mu^*$-, link-, or ordinary Zariski pairs of surface singularities in $\mathbb{C}^3$ consist of (possibly weighted) L\^e--Yomdin singularities. In this paper, we present an example of a $\mu$-Zariski pair involving surface singularities that are not of L\^e--Yomdin type.

\end{abstract}

\maketitle

\markboth{C. Eyral, M. Ishikawa, M. Oka, \"{O}. Turhan}{New $\mu$-Zariski pairs of surface singularities}

\section{Introduction}

Let $g_0(z_1,z_2,z_3)$ and $g_1(z_1,z_2,z_3)$ be polynomials in the complex variables $z_1,z_2,z_3$. Assume that $g_0$ and $g_1$ vanish at the origin $\mathbf{0}\in\mathbb{C}^3$, and that the corresponding surface germs $V(g_0)$ and $V(g_1)$ in $\mathbb{C}^3$ have isolated singularities at $\mathbf{0}$. We say that $g_0$ and $g_1$ form a \emph{Zariski pair of surface singularities} if they have the same monodromy zeta-function (in particular, the same Milnor number) at $\mathbf{0}$, and if the surface germs $V(g_0)$ and $V(g_1)$ have the same abstract topology but distinct embedded topologies.
The first example of such a pair was found by Artal Bartolo \cite{ArtalCRAS,ArtalMem} as a counterexample to a conjecture of Yau \cite{Yau}. In this example, the singularities of the polynomial germs $g_0$ and $g_1$ are \emph{superisolated} in the sense of Luengo \cite{Luengo}, and their projective tangent cones $C_0$ and $C_1$ form a \emph{Zariski pair of projective curves} in $\mathbb{P}^2$ with distinct \emph{Alexander polynomials}.
We recall that $(C_0,C_1)$ is a Zariski pair of curves in the complex projective plane $\mathbb{P}^2$ if there exist regular neighbourhoods $N_0$ and $N_1$ of $C_0$ and $C_1$, respectively, such that the pairs $(N_0, C_0)$ and $(N_1, C_1)$ are homeomorphic, whereas the pairs $(\mathbb{P}^2, C_0)$ and $(\mathbb{P}^2, C_1)$ are not (see \cite{Artal}).

By contrast, the situation in which the Alexander polynomials of $C_0$ and $C_1$ coincide is still far from being well understood. (For explicit examples of Zariski pairs of projective curves with the same Alexander polynomial, see, for instance, \cite{ArtalCarmona,O-1,O-2,Deg,EO4,EO5}.) An initial contribution in this direction was made by the first and third named authors in \cite{EO1}, where they constructed the first example of a \emph{$\mu^*$-Zariski pair of surface singularities} whose projective tangent cones may have identical Alexander polynomials. In that construction, the singularities forming the pair are of \emph{L\^e--Yomdin} type (in particular, the result applies to the class of superisolated singularities).
We recall that a $\mu^*$-Zariski pair is defined as a pair of surface singularities that have the same abstract topology, the same monodromy zeta-function, and the same Teissier $\mu^*$-invariant, but belong to distinct path-connected components of the $\mu^*$-constant stratum.

In \cite{EIO}, the first three named authors took a further step by constructing an example of a \emph{$\mu$-Zariski pair of surface singularities}. The notion of a $\mu$-Zariski pair is defined analogously to that of a $\mu^*$-Zariski pair, with the $\mu^*$-constant stratum replaced by the $\mu$-constant stratum. The examples considered in \cite{EIO} include L\^e--Yomdin singularities and, more generally, belong to the broader class of \emph{weighted} L\^e--Yomdin singularities.

It should be emphasized that the fact that $g_0$ and $g_1$ form a $\mu$- or $\mu^*$-Zariski pair does not, by itself, imply that $V(g_0)$ and $V(g_1)$ have distinct embedded topologies. Nevertheless, by a famous theorem of Teissier \cite{Teissier2}, being such a pair is a necessary condition for such a distinction to occur and therefore constitutes a substantial step toward a deeper understanding of the problem. 

Examples of pairs of L\^e--Yomdin singularities with the same monodromy zeta-function that are not connected by a $\mu$-constant path --- but whose links are non-homeomorphic --- were previously constructed by the third named author in \cite{O3}. The same reference also provides examples of so-called \emph{Zariski pairs of links}, namely pairs of germs whose links are homeomorphic and whose monodromy zeta-functions coincide. (By definition, a Zariski pair of links does not address whether the corresponding singularities lie in the same path-connected component of the $\mu$-constant stratum.)

In all cases considered so far --- whether $\mu$-, $\mu^*$-, link-, or ordinary Zariski pairs --- the surface germs involved consist of (possibly weighted) L\^e--Yomdin singularities. In this paper, we present an example of a $\mu$-Zariski pair whose singularities are not of L\^e--Yomdin type, thereby opening the way to the study of Zariski pair phenomena beyond the L\^e--Yomdin class.

\section{Statement of the theorem}

Let $f_0(z_1,z_2,z_3)$ and $f_1(z_1,z_2,z_3)$ be two reduced, irreducible, homogeneous polynomials of degree $d$ in the complex variables $z_1,z_2,z_3$, and let $C_0\equiv V_{\mathbb{P}^2}(f_0)$ and $C_1\equiv V_{\mathbb{P}^2}(f_1)$ denote the projective curves defined by $f_0$ and $f_1$, respectively, in the complex projective plane $\mathbb{P}^2$. We assume that $C_0$ and $C_1$ form a \emph{Zariski pair}. 
By a linear change of coordinates, we may assume that the singularities of $C_0$ and $C_1$ do not lie on the coordinate subspace $z_i=0$ for any $1\leq i\leq 3$, and that their defining polynomials $f_0$ and $f_1$ are \emph{convenient} and \emph{Newton non-degenerate} on the proper faces of their (common) Newton boundary (for the definitions of \emph{convenience} and \emph{Newton non-degeneracy}, we refer the reader to \cite{K}). They are, however, Newton degenerate on the (unique) top-dimensional face. 

Now, for $j\in\{0,1\}$, we consider the germ at $\mathbf{0}\in\mathbb{C}^3$ of the polynomial $g_j$ defined by
\begin{equation}\label{mainpoly}
g_j(z_1,z_2,z_3):=z_1^2f_j(z_1,z_2,z_3)+h_j(z_2,z_3),
\end{equation}
where $h_j(z_2,z_3)$ is a convenient, homogeneous polynomial in the variables $z_2$ and $z_3$ of degree $d+3$ such that the function
\begin{equation}\label{FF}
z_1^2f_j(0,z_2,z_3)+h_j(z_2,z_3)
\end{equation} 
is Newton non-degenerate. 
We suppose further that
\begin{equation}\label{pseudoLYA}
\mbox{Sing}(V_{\mathbb{P}^2}(f_j))\cap V_{\mathbb{P}^2}(h_j)=\emptyset.
\end{equation}
Here, $\mbox{Sing}(V_{\mathbb{P}^2}(f_j))$ stands for the singular locus of $V_{\mathbb{P}^2}(f_j)$, and  $V_{\mathbb{P}^2}(h_j)$ is the zero set of $h_j$ in $\mathbb{P}^2$.

\begin{remark}
Since $\mbox{Sing}(V_{\mathbb{P}^2}(z_1^2 f_j))$ includes a $1$-dimensional component, given by the equation $z_1=0$, the surface singularity $(V(g_j),\mathbf{0})$ in $\mathbb{C}^3$ is not of L\^e--Yomdin type (in particular, it is not superisolated).
\end{remark}

\begin{figure}[t]
\centering
\includegraphics[width=.4\linewidth]{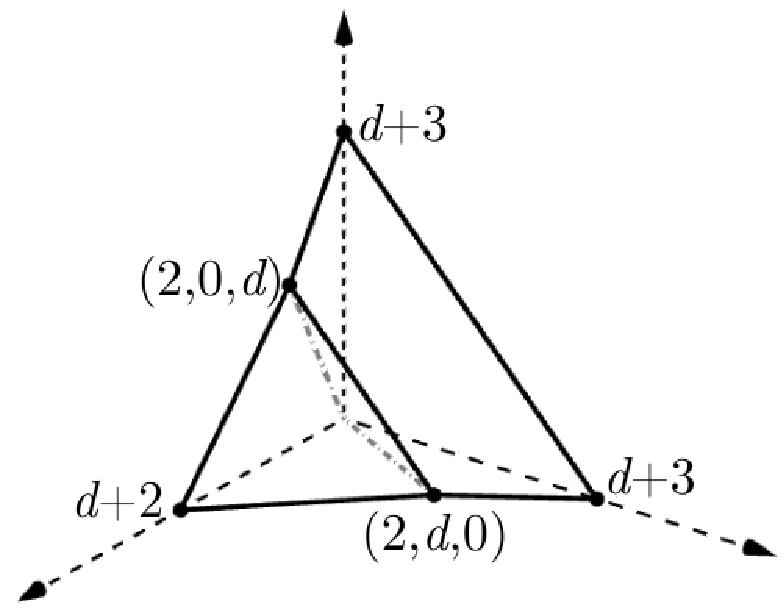}
\caption{Newton boundary $\Gamma$}
\label{Figure1}
\end{figure}

Clearly, $g_0$ and $g_1$ have the same Newton boundary, denoted by $\Gamma$, as illustrated in  Figure \ref{Figure1}.
Let $\Delta_0$ denote the $2$-dimensional face of $\Gamma$ defined by the vertices $(d+2,0,0)$, $(2,d,0)$ and $(2,0,d)$, and let $\Delta_1$ be that defined by the vertices $(2,d,0)$, $(2,0,d)$, $(0,d+3,0)$ and $(0,0,d+3)$. From our assumption it follows that $g_j$ is Newton non-degenerate on all proper faces of $\Delta_0$, while it is Newton degenerate on $\Delta_0$ itself. Besides, the above condition \eqref{FF} means that $g_j$ is Newton non-degenerate on $\Delta_1$ and its faces. Thus, $g_j$ is \emph{weakly almost Newton non-degenerate} in the sense of \cite[\S 2.7]{O3} (see also \cite[\S 3]{O2}), and since it is convenient, it has an isolated singularity at $\mathbf{0}$.

Now, consider the set $\mathcal{W}(\Gamma)$ consisting of all (germs at $\mathbf{0}$ of) polynomials $g(z_1,z_2,z_3)$ such that:
\begin{enumerate}
\item[$\bullet$]
$\Gamma(g)=\Gamma$, where $\Gamma(g)$ is the Newton boundary of $g$;
\item[$\bullet$]
$g$ is Newton non-degenerate on any face $\Delta$ of $\Gamma$ if $\Delta\not=\Delta_0$.
\end{enumerate} 
In particular, $g_0,g_1\in\mathcal{W}(\Gamma)$. Note that for any $g\in\mathcal{W}(\Gamma)$, the Newton principal part of $g$ is of the form
\begin{equation}\label{NPPofg}
z_1^2 f(z_1,z_2,z_3) + h(z_2,z_3),
\end{equation}
where $f(z_1,z_2,z_3)$ and $h(z_2,z_3)$ are convenient, homogeneous polynomials of degree $d$ and $d+3$, respectively, with $h(z_2,z_3)$ involving only the variables $z_2$ and $z_3$. Then, as above, $g$ is weakly almost Newton non-degenerate and has an isolated singularity at $\mathbf{0}$.

To ensure that $\mathcal{W}(\Gamma)$ is finite-dimensional, we additionally assume that there is an arbitrarily large fixed integer $n\geq d+3$ such that for any element $g$ in $\mathcal{W}(\Gamma)$, the degree of $g$ is less than or equal to $n$.

\begin{definition}\label{defwzpss}
We say that $g_0$ and $g_1$ form a \emph{$\mu$-Zariski pair of surface singularities in $\mathcal{W}(\Gamma)$} if the following three conditions are satisfied:
\begin{enumerate}
\item
$g_0$ and $g_1$ have the same monodromy zeta-function (in particular, the same Milnor number) at~$\mathbf{0}$;
\item
the surface germs $(V(g_0),\mathbf{0})$ and $(V(g_1),\mathbf{0})$ in $\mathbb{C}^3$, defined by $g_0$ and $g_1$ respectively, are homeomorphic;
\item
$g_0$ and $g_1$ cannot be joined by a $\mu$-constant continuous path in $\mathcal{W}(\Gamma)$ --- that is, they lie in distinct path-connected components of the $\mu$-constant strata of $\mathcal{W}(\Gamma)$.
\end{enumerate} 
\end{definition}

The next proposition is proved in section \ref{proof-pos}. 

\begin{proposition}\label{pos}
For any element $g\in\mathcal{W}(\Gamma)$, the following two assertions hold true.
\begin{enumerate}
\item
The polynomial $f$ that appears in \eqref{NPPofg} is reduced, that is, the plane projective  curve~$C\subseteq \mathbb{P}^2$ defined by $f$ has only isolated singularities.
\item
The Milnor number $\mu^{(2)}(g):=\mu(g^H)$ of the restriction $g^H$ of $g$ to a generic plane $H\subseteq \mathbb{C}^3$ passing through the origin is independent of the choice of $g$ within $\mathcal{W}(\Gamma)$. In particular, since the multiplicity at $\mathbf{0}$ of any $g\in\mathcal{W}(\Gamma)$ is constant (equal to $d+2$), the $\mu$-constant strata of $\mathcal{W}(\Gamma)$ coincide with the $\mu^*$-constant strata.
\end{enumerate}
\end{proposition}

We recall that the Teissier \emph{$\mu^*$-invariant} of an element $g\in\mathcal{W}(\Gamma)$ is the triple composed of the Milnor numbers $\mu(g)$ and $\mu^{(2)}(g)$ of $g$ and of its restriction to a generic plane of $\mathbb{C}^3$ passing through the origin, together with the multiplicity of $g$ at~$\mathbf{0}$ (see \cite{Teissier2}).
 
\begin{remark}
Note that, although the Milnor number $\mu^{(2)}(g)$ remains constant within $\mathcal{W}(\Gamma)$, the Milnor number $\mu(g)$ may change.
\end{remark}

The main result of this paper is stated as follows.

\begin{theorem}\label{mt}
If the singularities of the curves $C_0$ and $C_1$ are Newton non-degenerate in suitable local coordinates, then the polynomial function germs $g_0$ and $g_1$ defined by \eqref{mainpoly} form a $\mu$-Zariski pair  of surface singularities in $\mathcal{W}(\Gamma)$.
\end{theorem}

The proof of this theorem, inspired by \cite{EO1}, is presented in section \ref{sect-proof}.

\begin{remark}
The assumption \eqref{pseudoLYA} is crucial but only used to show that condition (1) of Definition \ref{defwzpss} holds. In general, without this assumption, the zeta-functions of $g_0$ and $g_1$ may differ as shown in section \ref{ex27} below. In contrast, to show that conditions (2) and (3) of the definition hold does not require \eqref{pseudoLYA}. In particular, we do not assume that elements $g$ in $\mathcal{W}(\Gamma)$ satisfy \eqref{pseudoLYA}.
\end{remark}

\begin{remark}
Provided that suitable analogues of assumptions \eqref{FF} and \eqref{pseudoLYA} are satisfied, Theorem \ref{mt} can be extended to polynomials of the form
\begin{equation*}
z_1^r f_j(z_1,z_2,z_3) + k_j(z_2,z_3),
\end{equation*}
where $f_j(z_1,z_2,z_3)$ is as above and $k_j(z_2,z_3)$ is a convenient homogeneous polynomial in $z_2$ and $z_3$ of degree $\delta\geq d+r+1$ with $r \geq 2$.
 Though the computation is more technical in this case (due to the presence of integral points in the interior of the face $\Delta_1$ if $\delta>d+r+1$), the underlying argument remains essentially the same as for $r=2$ and $\delta=d+3$. We leave the details to the reader. 

It is also worth noting that adding terms lying strictly above the Newton boundary of $z_1^r f_j + k_j$ preserves the validity of the result.
\end{remark}

We expect that, as in the case of superisolated singularities studied by Artal Bartolo in \cite{ArtalMem}, if, in addition to the assumptions of Theorem \ref{mt}, the Alexander polynomials of the curves $C_0$ and $C_1$ differ, then the germs $g_0$ and $g_1$ form an ordinary Zariski pair of surface singularities in $\mathbb{C}^3$ (i.e., the surface germs $(V(g_0),\mathbf{0})$ and $(V(g_1),\mathbf{0})$ have distinct embedded topologies in $\mathbb{C}^3$). This leads us to the following conjecture, in which we do not necessarily assume that the singularities of $C_0$ and $C_1$ are Newton non-degenerate.

\begin{conjecture}
If the Alexander polynomials of $C_0$ and $C_1$ are different, then the Jordan form of the monodromy of $g_0$ is different from that of $g_1$, and therefore the surface germs $(V(g_0),\mathbf{0})$ and $(V(g_1),\mathbf{0})$ have distinct embedded topologies in $\mathbb{C}^3$.
\end{conjecture}

A proof would require extending Artal Bartolo's result \cite[Th\'eor\`eme 1.6]{ArtalMem} on superisolated singularities to the class of singularities defined by \eqref{mainpoly}.

\section{Proof of Proposition \ref{pos}}\label{proof-pos}

To show the first item, we argue by contradiction. Suppose that $f$ is not reduced. Then it is of the form $f=h_1^2\cdot h_2\cdots h_r$, and therefore the face function $g_{\Delta}$ of $g$ corresponding to the edge $\Delta$ of $\Gamma$ defined by the vertices $(d+2,0,0)$ and $(2,d,0)$, for example, is written as
\begin{equation}\label{decoEFI}
\begin{aligned}
g_{\Delta}\equiv (z_1^2\, f)_{\Delta}
& = z_1^2\cdot (h_1)_{\Delta}^2\cdot (h_2)_{\Delta}\cdots (h_r)_{\Delta}\\
& = z_1^2\cdot (\alpha z_1+\beta z_2)^2\cdot\ell_2\cdots\ell_{d-2},
\end{aligned}
\end{equation}
where the $\ell_i$'s are linear forms (in $z_1,z_2$) and $\alpha,\beta\not=0$. (For the definition of a \emph{face function}, see \cite[\S 1.19]{K}.) Indeed, if $\alpha$ or $\beta$ were zero, then either the term $z_1^{d+2}$ or $z_1^2z_2^d$ would not appear in $g_{\Delta}$, which is a contradiction.
However, \eqref{decoEFI} implies that $g_{\Delta}$ is not Newton non-degenerate, contradicting the fact that $g\in\mathcal{W}(\Gamma)$.

\begin{figure}[t]
\centering
\includegraphics[width=.4\linewidth]{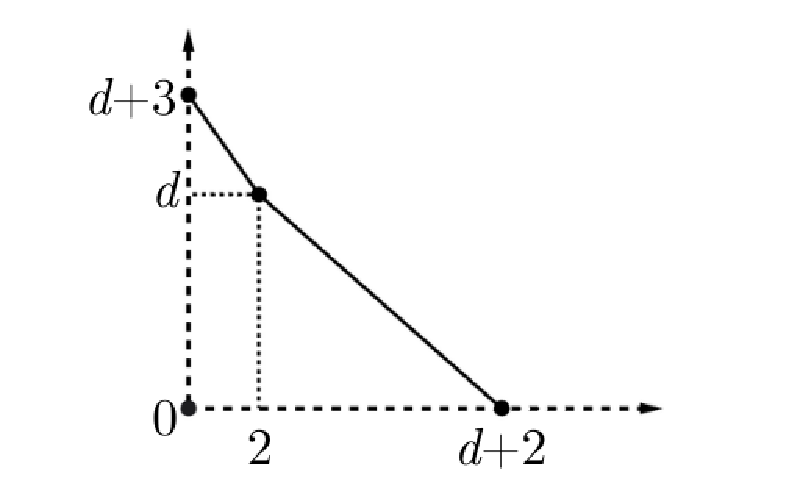}
\caption{Newton boundary $\Gamma(g^H)$}
\label{Figure3}
\end{figure}

Now, to show the second item, suppose that $H$ is given by the equation $z_3=az_1+bz_2$, where $a$ and $b$ are generic coefficients. 
As $f$ is reduced, so is its restriction $f^H$ to $H$, and therefore the Newton principal part of $g^H$---which coincides with the restriction to $H$ of the Newton principal part of $g$, namely
with $z_1^2 f^H+h^H$---is of the form 
\begin{align*}
c\, z_1^2\prod_{i=1}^d (z_1+a_{i}\,z_2)+h(z_2,az_1+bz_2),
\end{align*}
where $c\not=0$ and the $a_{i}$'s are non-zero and mutually distinct. Then, since the restriction $h^H$ of $h$ to $H$ is convenient, the Newton boundary of $g^H$ (see Figure \ref{Figure3}) --- and hence its Newton number --- is independent of $g$, and $g^H$ is convenient and Newton non-degenerate. Thus, by Kouchnirenko's theorem (see \cite[\S 1.10, Th\'eor\`eme~I]{K}), the Milnor number of $g^H$ is independent of the choice of $g$ within $\mathcal{W}(\Gamma)$. In fact, we have $\mu(g^H)=d^2+2d+2$.

\section{Proof of Theorem \ref{mt}}\label{sect-proof}

We must verify that conditions (1)–(3) of Definition \ref{defwzpss} are satisfied. 

\subsection{Condition (1)}\label{subsect-1}

Let $\Sigma^*$ be the regular simplicial cone subdivision of the dual Newton diagram $\Gamma^*(g_j)$ of $g_j$ ($j=0,1$) generated by the vertices $e_1=(1,0,0)$, $e_2=(0,1,0)$, $e_3=(0,0,1)$, $P={}^t(1,1,1)$, $Q={}^t(3,2,2)$ and $R=\frac{1}{2}(e_1+Q)={}^t(2,1,1)$ (see Figure \ref{Figure2}). Here, $P$ and $Q$ represent the weight vectors corresponding to the faces $\Delta_0$ and $\Delta_1$, respectively, while $R$ is introduced to ensure that the subdivision is regular. Then consider the toric modification $\pi\colon X\to\mathbb{C}^3$ associated with this subdivision, and for any vertex $S$ of $\Sigma^*$, let $\hat E(S)$ denotes the divisor of $\pi$ corresponding to $S$ (see \cite[Chap.~II, \S 1]{O1} or \cite[\S 2.5.1]{O3} for the definition). Note that $\hat E(e_i)$ is not an exceptional divisor. Then put
\begin{equation*}
E_j(S):=\hat E(S)\cap \tilde V(g_j),
\end{equation*}
where $\tilde V(g_j)$ is the strict transform of $V(g_j)$ by $\pi$. 

\begin{figure}[t]
\centering
\includegraphics[width=0.42\linewidth]{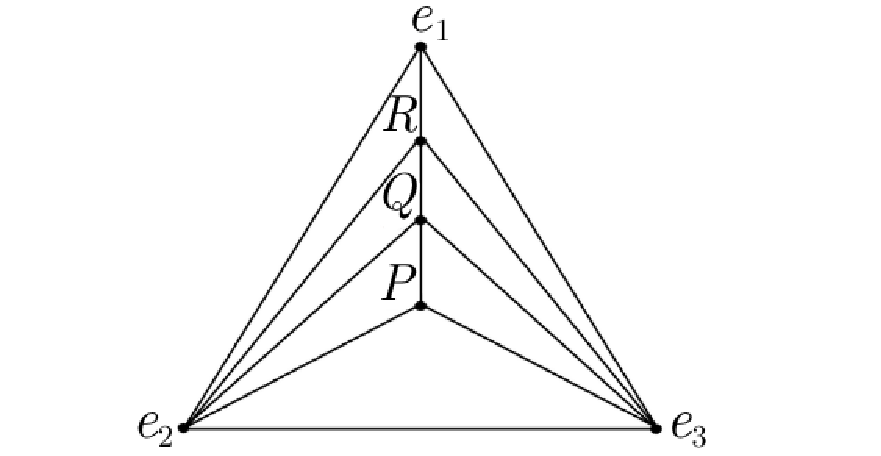}
\caption{Subdivision $\Sigma^*$}
\label{Figure2}
\end{figure}

Let $\mathbf{u}:=(u_1,u_2,u_3)$ denote the coordinates in the toric chart $\mathbb{C}^3_\sigma$ corresponding to the cone $\sigma:=\mbox{Cone}(P,e_2,e_3)$ of $\Sigma^*$ generated by $P$, $e_2$ and $e_3$. In this chart, the pull-back $\pi^*g_j$ of $g_j$ by $\pi$ is written as
\begin{align*}
\pi^*g_j (\mathbf{u}) = u_1^{d+2}(f_j(1,u_2,u_3)+u_1 h_j(u_2,u_3)).
\end{align*}
The first factor, $u_1^{d+2}$, corresponds to the exceptional divisor $\hat E(P)$, while the second one represents the strict transform $\tilde V(g_j)$.
Note that $\tilde V(g_j)$ is non-singular outside of $\hat E(P)$, so that $E_j(Q)$ and $E_j(R)$ are smooth and $E_j(Q)$ transversely intersects both $E_j(R)$ and $E_j(P)$. On the other hand, $E_j(P)$ has a finite number of isolated singularities, which are given by the singular points of the curve $C_j$. 
Now, from our assumptions, the function $g_{j}$ is weakly almost Newton non-degenerate. Thus, by \cite[Theorem 9]{O3} or \cite[Lemma 3.2, Remark 3.6 and Theorem 3.7]{O2}, the monodromy zeta-function $\zeta_{g_j,\mathbf{0}}(t)$ of $g_j$  at $\mathbf{0}$ is given by
\begin{equation}\label{OF}
\zeta_{g_j,\mathbf{0}}(t)=\zeta(t)\times (1-t^{d+2})^{\mu^{\text{tot}}(C_j)} \times \prod_{\substack{\rho\in\text{Sing}(C_j)}} \zeta_{\pi^*g_j,\rho}(t),
\end{equation}
where $\zeta(t)$ is the monodromy zeta-function of a Newton non-degenerate function whose Newton boundary is $\Gamma$, and $\mu^{\text{\tiny tot}}(C_j)$ is the total Milnor number of $C_j$ (i.e., the sum of the local Milnor numbers at the singular points of $C_j$).
The term on the far right of \eqref{OF} is the product, over the singular points $\rho$ of $C_j$, of the monodromy zeta-functions of $\pi^* g_j$ at each of these points.
By Varchenko's theorem (see \cite[Theorem (4.1)]{V}), we have
\begin{equation}\label{OF1}
\zeta(t)=(1-t^{d+2})^{-d^2+2d-1}(1-t^{d+3})^{d+1}(1-t^{2d+6})^{-2d-1}.
\end{equation}
(The first factor arises from the face $\Delta_0$ and its subfaces, whereas the second and third ones originate from the remaining faces.)
Moreover, since $(C_0,C_1)$ is a Zariski pair, the total Milnor numbers $\mu^{\text{tot}}(C_0)$ and $\mu^{\text{tot}}(C_1)$ coincide. Consequently, by \eqref{OF}, in order to establish the equality $\zeta_{g_0,\mathbf{0}}(t)=\zeta_{g_1,\mathbf{0}}(t)$,  it suffices to show that for any singular points ${\rho}_0$ and ${\rho}_1$ of $C_0$ and $C_1$, respectively, chosen so that the germs $(C_0,\rho_0)$ and $(C_1,\rho_1)$ are topologically equivalent, the zeta-functions $\zeta_{\pi^*g_0,\rho_0}(t)$ and $\zeta_{\pi^*g_1,\rho_1}(t)$ coincide. This latter property will follow from the Newton non-degeneracy of the singularities of the curves $C_0$ and $C_1$, which implies that the polynomials $g_0$ and $g_1$ are in fact \emph{almost Newton non-degenerate} in the sense of \cite[\S 3]{O2} --- and not merely \emph{weakly} almost Newton non‑degenerate. More precisely, we have the following lemma.

\begin{lemma}\label{cannd}
For any $j\in\{0,1\}$ and any singular point $\rho\in\mbox{\emph{Sing}}(E_j(P))\simeq \mbox{\emph{Sing}}(C_j)$, there are local coordinates around $\rho$ in which $\pi^*g_j$ is Newton non-degenerate. In other words, the polynomial $g_j$ is almost Newton non-degenerate.
\end{lemma}

\begin{proof}
Let $(0,\rho_{2},\rho_{3})$ be the coordinates of $\rho$ in the toric chart $(\mathbb{C}^3_\sigma,\mathbf{u})$. Consider the coordinates $(u_1,u_2-\rho_{2},u_3-\rho_{3})$ centred at $\rho$. As the germ $(C_j,\rho)$ is defined by a Newton non-degenerate function, there exists an analytic coordinate chart $(\mathcal{U}(\rho),\mathbf{v}=(v_1,v_2,v_3))$ of $X$ at $\rho$ (i.e., $\mathbf{v}(\rho)=\mathbf{0}$) such that:
\begin{enumerate} 
\item
$v_{1}=u_1$ and $\mathbf{v}':=(v_{2},v_{3})$ is an analytic coordinate change of $(u_2-\rho_{2},u_3-\rho_{3})$, that is, there exists $\Phi=(\phi_2,\phi_3)\in\mbox{Aut}(\mathbb{C}^{2})$ such that $u_2-\rho_{2}=\phi_2(\mathbf{v}')$, $u_3-\rho_{3}=\phi_3(\mathbf{v}')$, and $\Phi(\mathbf{0})=(\phi_2(\mathbf{0}),\phi_3(\mathbf{0}))=(0,0)$; 
\item
the defining polynomial 
\begin{equation*}
\bar f_j(\mathbf{v}'):=f_j(1,\rho_{2}+\phi_2(\mathbf{v}'),\rho_{3}+\phi_3(\mathbf{v}'))
\end{equation*} 
of $E_j(P)$ near $\rho$ is Newton non-degenerate.
\end{enumerate}
Hereafter, coordinates $\mathbf{v}=(v_1,v_2,v_3)$ satisfying condition (1) will be called \emph{admissible} coordinates. 
In $(\mathcal{U}(\rho),\mathbf{v})$, the pull-back $\pi^*g_j$ is written as
\begin{equation*}
\pi^*g_j (\mathbf{v}) = v_{1}^{d+2} (\bar f_j(\mathbf{v}')+ v_{1} h_j(\rho_{2}+\phi_2(\mathbf{v}'),\rho_{3}+\phi_3(\mathbf{v}'))).
\end{equation*}
By assumption \eqref{pseudoLYA}, we have $h_j(\rho_{2},\rho_{3})\not=0$, so that the Newton principal part of $\pi^*g_j$ in the coordinates $\mathbf{v}=(v_1,v_{2},v_{3})$ coincides with that of
\begin{equation*}
v_{1}^{d+2} (\bar f_j(\mathbf{v}')+ v_{1}h_j(\rho_{2},\rho_{3})).
\end{equation*}
(Assumption \eqref{pseudoLYA} is used solely to ensure that the linear term $v_{1}h_j(\rho_{2},\rho_{3})$ appears in the expression above with a non-zero coefficient.) Then,
since $\bar f_j$ is Newton non-degenerate in the coordinates $\mathbf{v}'=(v_{2},v_{3})$, we deduce that $\pi^*g_j$ is Newton non-degenerate in the coordinates $\mathbf{v}=(v_1,v_{2},v_{3})$, and therefore $g_j$ is almost Newton non-degenerate.
\end{proof}

Now, let ${\rho}_0$ and ${\rho}_1$ be singular points of $C_0$ and $C_1$, respectively, chosen so that the germs $(C_0,\rho_0)$ and $(C_1,\rho_1)$ are topologically equivalent. By Lemma \ref{cannd}, there are local admissible coordinates $\mathbf{v}=(v_1,v_2,v_3)$ and $\mathbf{w}=(w_1,w_2,w_3)$ near ${\rho}_0$ and ${\rho}_1$, respectively, such that the Newton principal parts of $\pi^*g_0$ and $\pi^*g_1$ are identical (up to coefficients) to those of
\begin{equation*}
v_{1}^{d+2} (\bar f_0(\mathbf{v}')+ v_{1})
\quad\mbox{and}\quad
w_{1}^{d+2} (\bar f_1(\mathbf{w}')+ w_{1}),
\end{equation*}
where $\bar f_0(\mathbf{v}')$ and $\bar f_1(\mathbf{w}')$ are Newton non-degenerate. Since $(C_0,\rho_0)$ and $(C_1,\rho_1)$ are topologically equivalent, we may further assume that the Newton diagrams of $\bar f_0$ and $\bar f_1$, and hence those of $\pi^*g_0$ and $\pi^*g_1$, coincide up to an analytic change of coordinates (see \cite[Remark 3.17, Proposition 4.17 \& Corollary 4.18]{BMP} and \cite[Theorem~5.5.8]{W}). Then, since $\pi^*g_0$ and $\pi^*g_1$ are Newton non-degenerate, it follows from Varchenko's result (see \cite[Theorem~(4.1)]{V}) that the zeta-functions $\zeta_{\pi^*g_0,\rho_0}(t)$ and $\zeta_{\pi^*g_1,\rho_1}(t)$ are equal.

\subsection{Condition (2)}\label{subsect-2}

Now, let us show that the surface germs $V(g_0)$ and $V(g_1)$ at $\mathbf{0}$ --- equivalently, the links $K_{g_0}:=V(g_0)\cap\mathbb{S}^5$ and $K_{g_1}:=V(g_1)\cap\mathbb{S}^5$ --- are homeomorphic. Here, the proof proceeds along lines similar to those presented in \cite[Theorem~25 \& Remark~26]{O3}. Namely we show that, when regarded as graph manifolds, the links $K_{g_0}$ and $K_{g_1}$ share the same plumbing graph. We then invoke a classical theorem of Waldhausen \cite{Waldhausen} and Neumann \cite{Neumann}, which asserts that any plumbing graph determines a unique graph manifold up to homeomorphism. In particular, if $V(g_0)$ and $V(g_1)$ admit resolutions with the same dual resolution graph, then they are homeomorphic.\footnote{As is standard, plumbing graphs and dual resolution graphs are, by default, \emph{weighted} graphs.}

For $j\in\{0,1\}$, let $\rho_{j,1},\ldots\rho_{j,s}$ denote the singularities of $E_j(P)$ (i.e., the singular point of the curve $C_j$). As above, for each $\rho_{j,i}$, there are local admissible coordinates $\mathbf{v}=(v_1,v_2,v_3)$ near $\rho_{j,i}$ such that the Newton principal part of $\pi^*g_j$ coincides (up to coefficients) with that of
\begin{equation*}
v_{1}^{d+2} (\bar f_j(v_2,v_3)+ v_{1}),
\end{equation*}
where $\bar f_0(v_2,v_3)$ and $\bar f_1(v_2,v_3)$ are given by the same fixed, normal form of the Newton non-degenerate singularity $(C_0,\rho_{0,i})\simeq (C_1,\rho_{1,i})$.
To resolve the singularities $\rho_{j,i}$, we proceed further toric modifications $\omega_{j,i}\colon Y_{j,i}\to X$ at each $\rho_{j,i}$ with respect to the same regular simplicial cone subdivision for $j=0$ and $j=1$. Denote by $\hat E_{j,i,1},\ldots, \hat E_{j,i,r_i}\subseteq Y_{j,i}$ the corresponding exceptional divisors. Note that $r_i$ does not depend on $j$. Now, let
\begin{equation*}
\Pi_j\colon Y_j\xrightarrow{\omega_j}X\xrightarrow{\pi}\mathbb{C}^3
\end{equation*}
denote the resolution of $g_j$ obtained by composing the $\omega_{j,i}$'s with $\pi$. Here, $Y_j$ is the canonical gluing of the union of the $Y_{j,i}$'s, and $\omega_j\colon Y_j\to X$ denotes the union of the toric modifications $\omega_{j,i}$. Let $\tilde V_{Y_j}(g_j)$ be the strict transform of $V(g_j)$ by $\Pi_j$. Then, put 
\begin{equation*}
E_{j,i,k}:=\hat E_{j,i,k}\cap \tilde V_{Y_j}(g_j), 
\end{equation*}
and write
\begin{align*}
& E_{Y_j}(S):=\omega_j^{-1}(E_j(S)) \mbox{ for } S=Q\mbox{ or }R,\\
& E_{Y_j}(P):=\overline{\omega_j^{-1}(E_j(P)\setminus\{\rho_{j,1},\ldots,\rho_{j,s}\})},
\end{align*}
for the strict transforms of $E_j(S)$ and $E_j(P)$ by $\omega_j$, respectively. (Here, the bar denotes the topological closure in $Y_j$.) 
Note that 
\begin{align*}
& E_{Y_j}(S)\simeq E_j(S)\mbox{ for } S=Q\mbox{ or }R,\\
& E_{Y_j}(P)\bigg\backslash \bigcup_{i,k}E_{j,i,k}\simeq E_j(P)\setminus\{\rho_{j,1},\ldots,\rho_{j,s}\},
\end{align*}
where the union is taken over all $1\leq i\leq s$ and $1\leq k\leq r_i$.

As $(C_0,C_1)$ is a Zariski pair, $C_0$ and $C_1$ have the same singularities, and therefore both surfaces $V(g_0)$ and $V(g_1)$ get exactly the same configuration of exceptional divisors 
\begin{equation}\label{lod}
E_{Y_j}(P), \ E_{Y_j}(Q), \ E_{Y_j}(R), \ E_{j,i,1},\ldots, E_{j,i,r_i}
\end{equation}
 in $Y_j$. 
Each of the divisors $E_{Y_0}(R)$ and $E_{Y_1}(R)$ is a disjoint union of $d$ rational spheres $\mathbb{S}^2$, which implies that they are homeomorphic. Concerning the smooth divisors $E_{Y_0}(Q)$ and $E_{Y_1}(Q)$, since $g_j$ is Newton non-degenerate on $\Delta_1$ (i.e., on the face associated with the weight vector $Q$) and on its subfaces, it follows from the Khovanskii--Kouchnirenko--Oka theorem \cite{Kh,K,O4} (see also \cite[Chap.~III, Lemma (5.1) \& Chap.~IV, Theorem (3.1)]{O1}) that the Euler characteristic of $E_{Y_j}^*(Q)$, and hence that of $E_{Y_j}(Q)$,
 is completely determined by $\Delta_1$, so that $\chi(E_{Y_0}(Q))=\chi(E_{Y_1}(Q))$. Here, $E_{Y_j}^*(Q)$ is defined by
\begin{equation*}
E_{Y_j}^*(Q):=E_{Y_j}(Q)\bigg\backslash\bigcup_{S=P,R,e_2,e_3} E_{Y_j}(Q)\cap E_{Y_j}(S).
\end{equation*}
In particular, this implies that $E_{Y_0}(Q)$ and $E_{Y_1}(Q)$ have the same genus, and consequently, that they are homeomorphic.
Finally, concerning the \emph{irreducible} divisors $E_{Y_0}(P)$ and $E_{Y_1}(P)$, the Euler characteristic $\chi(E_{Y_j}(P))$ is given by
\begin{equation*}
\chi(E_{Y_j}(P))=3d-d^2+\mu^{\text{tot}}(C_j)+\sum_{i=1}^s (\nu_{j,i}-1),
\end{equation*}
where $\nu_{j,i}$ is the number of local branches of $C_j$ at $\rho_{j,i}$, and since $C_0$ and $C_1$ have the same singularities, it follows that $\chi(E_{Y_0}(P))=\chi(E_{Y_1}(P))$. 
Now, since they are irreducible, they also have the same genus, and therefore they are homeomorphic.

Again, because the singularities of $C_0$ and $C_1$ are identical, we similarly verify that $E_{0,i,k}$ and $E_{1,i,k}$ are homeomorphic for any $1\leq i\leq s$ and $1\leq k\leq r_i$.

Now, to show that the resolutions
\begin{equation*}
\Pi_0\vert_{\tilde V_{Y_0}(g_0)} \colon \tilde V_{Y_0}(g_0)\to V(g_0)
\quad\mbox{and}\quad
\Pi_1\vert_{\tilde V_{Y_1}(g_1)} \colon \tilde V_{Y_1}(g_1)\to V(g_1) 
\end{equation*}
have the same dual resolution graph, it remains to verify that for any $\ell,\ell'$ the intersection numbers 
$E_{0,\ell}\cdot E_{0,\ell'}$ and $E_{1,\ell}\cdot E_{1,\ell'}$
are identical, where $E_{j,\ell}$ and $E_{j,\ell'}$ denote any of the irreducible components of the divisors that appear in the list \eqref{lod} (in particular, each divisor $E_{j,i,k}$, $1\leq k\leq r_i$, is a union of such irreducible pieces $E_{j,\ell}$). Such an equality is clear for $\ell\not=\ell'$. For the self-intersection numbers $E_{j,\ell}^2$, we use the property that for any rational function 
\begin{equation*}
\varphi_j \colon \tilde V_{Y_j}(g_j) \to \mathbb{P}^1
\end{equation*} 
on the smooth variety $\tilde V_{Y_j}(g_j)$, the intersection number $(\varphi_j) \cdot E_{j,\ell}$ must be zero, where $(\varphi_j)$ denotes the divisor defined by $\varphi_j$ (see, e.g., \cite[Theorem~2.6]{Laufer}). In particular, by taking $\varphi_j:=\Pi_j^*z_2\vert_{\tilde V_{Y_j}(g_j)}$, we obtain
\begin{equation}\label{ae25}
(\varphi_j)=\sum_{\ell} m_{j,\ell}\, E_{j,\ell},
\end{equation}
so that, for instance, the self-intersection number $E_{Y_j}(P)^2$ of the irreducible divisor $E_{Y_j}(P)$ satisfies
\begin{equation}\label{ae25bis}
E_{Y_j}(P)^2 = - \frac{1}{m_j(P)} \Bigg( \sum_{\{\ell\, \mid\, E_{j,\ell}\not= E_{Y_j}(P)\}} m_{j,\ell}\, E_{j,\ell}\cdot E_{Y_j}(P)\Bigg).
\end{equation}
The sum in \eqref{ae25} is taken over all the irreducible components $E_{j,\ell}$ of the divisors that appear in (4.3), while that in \eqref{ae25bis} is taken over all these  components except $E_{Y_j}(P)$. The integer $m_{j,\ell}$ (respectively, $m_j(P)$) denotes the multiplicity of $\varphi_j$ along $E_{j,\ell}$ (respectively, $E_{Y_j}(P)$).
Thus, once again, since $C_0$ and $C_1$ have the same singularities, it follows that $E_{Y_0}(P)^2=E_{Y_1}(P)^2$.

The relations $E_{0,\ell}^2=E_{1,\ell}^2$ for the other irreducible pieces $E_{j,\ell}$'s are obtained by similar arguments.

Altogether, the restrictions of $\Pi_0$ and $\Pi_1$ to $\tilde V_{Y_0}(g_0)$ and $\tilde V_{Y_1}(g_1)$, respectively, have the same dual resolution graph, and we conclude by invoking Waldhausen--Neumann's theorem \cite{Waldhausen,Neumann}.

\begin{remark}
Note that, in \S \ref{subsect-1}, to establish the equality $\zeta_{g_0,\mathbf{0}}(t)=\zeta_{g_1,\mathbf{0}}(t)$, we relied solely on the \emph{almost Newton non-degeneracy theory} developed in \cite{O2,O3}. In particular, no information about the dual resolution graphs of $\Pi_0\colon\tilde V_{Y_0}(g_0)\to V(g_0)$ and $\Pi_1\colon\tilde V_{Y_1}(g_1)\to V(g_1)$ was required. By contrast, in \S \ref{subsect-2}, proving that the surface germs $V(g_0)$ and $V(g_1)$ have the same abstract topology necessitated the use of dual resolution graph theory.
\end{remark}

\subsection{Condition (3)}

It remains to show that $g_0$ and $g_1$ lie in different path-connected components of the $\mu$-constant strata of $\mathcal{W}(\Gamma)$. We argue by contradiction. Assume that there exists a $\mu$-constant continuous path 
\begin{equation*}
\gamma\colon s\in I:=[0,1]\to g_s:=\gamma(s)\in\mathcal{W}(\Gamma)
\end{equation*}
 such that $\gamma(0)=g_0$ and $\gamma(1)=g_1$. Then, without loss of generality, we may assume that $\gamma$ is a \emph{piecewise complex-analytic} path.
Indeed, by \cite[\S 3]{EO2}, the $\mu$-constant stratum of an isolated surface singularity is a \emph{constructible} set, and taking its intersection with $\mathcal{W}(\Gamma)$ (which is also constructible) preserves this structure. The stated property then follows from an argument similar to that used in the proof of \cite[Theorem 5.4 \& Proposition~5.2]{EO2}. 

As in \eqref{NPPofg}, the Newton principal part of $g_s$ is written as
\begin{equation*}
z_1^2 f_s(z_1,z_2,z_3) + h_s(z_2,z_3),
\end{equation*}
where $f_s(z_1,z_2,z_3)$ and $h_s(z_2,z_3)$ are convenient, homogeneous polynomials of degree $d$ and $d+3$, respectively, and the plane projective curve~$C_{s}\subseteq \mathbb{P}^2$ defined by $f_s$ has only isolated singularities (see Proposition \ref{pos}). 

\begin{lemma}\label{zmcc2}
The total Milnor number $\mu^{\text{\emph{tot}}}(C_s)$ is independent of $s$.
\end{lemma}

\begin{proof}
Since the Milnor number of $g_s$ is independent of $s$, so is the zeta-function $\zeta_{g_s,\mathbf{0}}(t)$ (see \cite[Lemma~12]{O3}). Moreover, like in \eqref{OF}, since $g_s$ is weakly almost Newton non-degenerate, its zeta-function at $\mathbf{0}$ is given by
\begin{equation*}
\zeta_{g_s,\mathbf{0}}(t)=\zeta(t)\times (1-t^{d+2})^{\mu^{\text{tot}}(C_s)} \times \prod_{\substack{\rho\in\text{Sing}(C_s)}} \zeta_{\pi^*g_s,\rho}(t),
\end{equation*}
where $\zeta(t)$ is as in \eqref{OF1} (see \cite{O2,O3}),
and for each $\rho\in\text{Sing}(C_s)$, the exceptional divisors of a suitable resolution of $\pi^*\! g_s$ at $\rho$ all have multiplicity greater than $d+2$. 
Since $\zeta_{g_s,\mathbf{0}}(t)$ is uniquely written as a product $\prod (1-t^{d_i})^{\nu_i}$ with the $d_i$'s all distinct and the $\nu_i$'s non-zero (see \cite[Th\'eor\`eme 3]{AC}), it follows that the factor $\zeta(t)\times (1-t^{d+2})^{\mu^{\text{tot}}(C_s)}$, and hence the total Milnor number $\mu^{\text{tot}}(C_s)$, is independent of $s$.
\end{proof}

\begin{lemma}\label{nbs}
The family of curves $C_s$, $s\in I$, has no bifurcation of singularities.
\end{lemma}

\begin{proof}
We argue by contradiction. Suppose there exists $s_0$ such that the family has a bifurcation of singularities in a small ball $B$ centred at a singular point $\rho_0$ of $C_{s_0}$.
Then, by a theorem of Lazzeri \cite{Lazzeri} and L\^e \cite[Th\'eor\`eme~B]{Le} (see also \cite{B}), for $s\not=s_0$ near $s_0$, we have
\begin{equation*}\label{plemma-esmn}
\sum_{\rho\in B\cap \text{Sing}(C_s)}\mu(C_s,\rho) 
< \mu(C_{s_0},\rho_0),
\end{equation*}
where $\mu(C_s,\rho)$ and $\mu(C_{s_0},\rho_0)$  are the Milnor numbers of  $C_s$ and $C_{s_0}$ at $\rho$ and $\rho_0$, respectively.
It folows that $\mu^{\text{\tiny tot}}(C_{s})<\mu^{\text{\tiny tot}}(C_{s_0})$, which contradicts Lemma \ref{zmcc2}. 
\end{proof}

To complete the proof of the theorem, it suffices to observe that, combined with another well known theorem of L\^e \cite{Le2} (see also the discussion in \cite[pp.~17--18 and p.~121]{Dimca}), Lemma \ref{nbs} implies that the topological type of the pair $(\mathbb{P}^2,C_s)$ is independent of $s$, and therefore $(C_0,C_1)$ is not a Zariski pair --- a contradiction.

\section{Effect of dropping the assumption \eqref{pseudoLYA}}\label{ex27}

In this section, we show that if the assumption \eqref{pseudoLYA} fails for at least one of the functions $g_0$ or $g_1$, then the zeta-functions $\zeta_{g_0,\mathbf{0}}(t)$ and $\zeta_{g_1,\mathbf{0}}(t)$ may differ. Consider, for example, the polynomials $g_0$ and $g_1$ defined by
\begin{align*}
& g_0:=z_1^2(z_1+z_2-2z_3)(z_1+3z_2-4z_3)+z_2^5+z_3^5,\\
& g_1:=z_1^2(z_1+z_2-2z_3)(z_1+3z_2-4z_3)+z_2^5-z_3^5,
\end{align*}
which differ only in the coefficient $\pm 1$ of the last monomial $\pm z_3^5$.
The Newton boundary $\Gamma:=\Gamma(g_0)=\Gamma(g_1)$ has two $2$-dimensional faces: a triangular face given by the vertices $(4,0,0)$, $(2,2,0)$, $(2,0,2)$, and a rectangular one defined by $(2,2,0)$, $(2,0,2)$, $(0,5,0)$ and $(0,0,5)$. Both $g_0$ and $g_1$ are in $\mathcal{W}(\Gamma)$. The curve $C:=C_0\equiv C_1$, defined by $f_0 \equiv f_1 := (z_1 + z_2 - 2 z_3)(z_1 + 3 z_2 - 4 z_3)$, possesses a single node at $[1:1:1]$. This point does not lie on $\{h_0 := z_2^5 + z_3^5=0\}$, but belongs to $\{h_1 := z_2^5 - z_3^5=0\}$,
so that $g_0$ satisfies assumption \eqref{pseudoLYA}, whereas $g_1$ does not.
 As above, using \cite{O2,O3}, we easily check that the zeta-functions $\zeta_{g_0,\mathbf{0}}(t)$ and $\zeta_{g_1,\mathbf{0}}(t)$ of the weakly almost Newton non-degenerate functions $g_0$ and $g_1$ are given by
\begin{align*}
& \zeta_{g_0,\mathbf{0}}(t)=(1-t^5)^{2}(1-t^{10})^{-5},\\
& \zeta_{g_1,\mathbf{0}}(t)=(1-t^5)^{3}(1-t^{10})^{-5}(1-t^6)^{-1},
\end{align*}
so that the Milnor numbers $\mu(g_0)=-1-\deg \zeta_{g_0,\mathbf{0}}(t)$ and $\mu(g_1)=-1-\deg \zeta_{g_1,\mathbf{0}}(t)$ are $39$ and $40$, respectively. In other words, condition (1) of Definition \ref{defwzpss} fails. 

The reason for this failure is that, under the toric modification described in \S\ref{subsect-1}, the pullback $\pi^{*}g_{0}$ at the singular point of $C$ takes the form $v_1^2(v_2^2+v_3^2)+cv_1+\cdots$, where the linear term $c v_{1}$ arises from \eqref{pseudoLYA}, while the pullback $\pi^{*}g_{1}$ is given by the expression $v_1^2(v_2^2+v_3^2)+v_1(v_2+v_3)+\cdots$, without any linear term in $v_1$. (Here, the dots ``$\cdots$'' indicate higher-order terms.)

\section*{Acknowledgments}

This research was supported by the Narodowe Centrum Nauki under the grant number
2023/49/B/ST1/00848.

\bibliographystyle{amsplain}

\end{document}